\UseRawInputEncoding
\documentclass{cmslatex}
\usepackage[paperwidth=7in, paperheight=10in, margin=.875in]{geometry}
 \usepackage[backref,colorlinks,linkcolor=red,anchorcolor=green,citecolor=blue]{hyperref}
\usepackage{amsfonts,amssymb}
\usepackage{amsmath}
\usepackage{graphicx}
\usepackage{cite}
\usepackage{enumerate}
\renewcommand{\theequation}{\arabic{section}.\arabic{equation}}
\def\S{\mathcal{S}}

\def\Re{\mathbb{R}}

\def\I{\mathcal{I}}

\def\A{{\mathcal{A}}}
\def\B{\mathcal{B}}

\def\CC{\mathbb{C}}

\def\D{\mathcal{D}}
\def\U{\mathcal{U}}
\def\V{\mathcal{V}}

\def\S{\mathcal{S}}

\def\0{{\bf 0}}

   \allowdisplaybreaks
\begin{document}
 \title{ST-SVD Factorization and s-Diagonal Tensors\thanks{To appear in: Communications in Mathematical Sciences.}}


          \author{Chen Ling\thanks{Department of Mathematics, Hangzhou Dianzi University, Hangzhou, 310018, China, (macling@hdu.edu.cn). This author's work was supported by Natural Science Foundation of China (No. 11971138) and Natural Science Foundation of Zhejiang Province (No. LY19A010019, LD19A010002).}
          \and Jinjie Liu\thanks{School of Mathematical Sciences, Shanghai Jiao Tong University, Shanghai, 200240, China, (jinjie.liu@sjtu.edu.cn). This author's work was supported by Natural Science Foundation of China (No. 11801479, No. 12001366).}
          \and Chen Ouyang\thanks{School of Computer Science and Technology, Dongguan University of Technology, Dongguan, 523000, China, (oych26@163.com). This author's work was supported by Natural Science Foundation of China (No.11971106) and Guangdong Universities’ Special Projects in Key Fields of Natural Science
(No. 2019KZDZX1005).}
          \and Liqun Qi\thanks{Department of Mathematics, Hangzhou Dianzi University, Hangzhou, 310018, China; Department of Applied
    Mathematics, The Hong Kong Polytechnic University, Hung Hom,
    Kowloon, Hong Kong, China, (liqun.qi@polyu.edu.hk). Corresponding Author.}}

         \pagestyle{myheadings} \markboth{ST-SVD Factorization}{Chen Ling, Jinjie Liu, Chen Ouyang and Liqun Qi} \maketitle

          \begin{abstract}
               A third order real tensor is mapped to a special f-diagonal tensor by going through Discrete Fourier Transform (DFT), standard matrix SVD and inverse DFT.  We call such an f-diagonal tensor an s-diagonal tensor.   An f-diagonal tensor is an s-diagonal tensor if and only if it is mapped to itself in the above process.  The third order tensor space is partitioned to orthogonal equivalence classes.   Each orthogonal equivalence class has a unique s-diagonal tensor.  Two s-diagonal tensors are equal if they are orthogonally equivalent.   Third order tensors in an orthogonal equivalence class have the same tensor tubal rank and T-singular values. Four meaningful necessary conditions for s-diagonal tensors are presented.   Then we present a set of sufficient and necessary conditions for s-diagonal tensors.  Such conditions involve a special complex number.   In the cases that the dimension of the third mode of the considered tensor is $2, 3$ and $4$, we present direct sufficient and  necessary conditions which do not involve such a complex number.
          \end{abstract}
\begin{keywords}   T-SVD factorization; s-diagonal tensor; f-diagonal tensor; necessary conditions; sufficient and necessary conditions
\end{keywords}

 \begin{AMS} 15A69; 15A18
\end{AMS}
          \section{Introduction}\label{intro}
Matrix SVD factorization is a fundamental tool in numerical linear algebra \cite{GV13}.  For a real $m \times n$ matrix $A$, it is decomposed to the product of an $m \times m$ orthogonal matrix $U$, an $m \times n$ diagonal matrix $S$, and an $n \times n$ orthogonal matrix $V$:
\begin{equation} \label{e1.1}
A = USV^\top.
\end{equation}
Here, $S$ needs to be a nonnegative diagonal matrix.   Then the diagonal entries of $S$ are singular values of $A$.   If the diagonal entries of $S$ are in the standard nonincreasing order, then we may regard (\ref{e1.1}) as a standard SVD of $A$ \cite{GV13}.

The matrix SVD factorization (\ref{e1.1}) was extended to third order tensors as T-SVD factorization by Kilmer and Martin \cite{KM11, KMP08}.   The T-SVD factorization has been found wide applications in engineering and tensor computation \cite{CXZ20, KBHH13, LHPQ21, LYQX20, Lu18, MQW20, MQW21, QY21, SHKM14, SNZ21, XCGZ21, XCGZ21a, YHHH16, ZSKA18, ZEAHK14, ZA17, ZHW21, ZLLZ18}.    In T-SVD factorization, an $m \times n \times p$ third order real tensor $\A$ is decomposed to the tensor product of an $m \times m \times p$ orthogonal tensor $\U$, an $m \times n \times p$ real f-diagonal tensor $\S$, and an $n \times n \times p$ orthogonal tensor $\V$:
\begin{equation} \label{e1.2}
\A = \U * \S * \V^\top,
\end{equation}
where $*$ is the tensor-tensor product, or t-product for short.  An $m \times n \times p$ tensor is called an f-diagonal tensor\cite{KM11} if all of its frontal slices $A^{(k)}$ are diagonal matrices, for $k = 1, 2, \ldots, p$.  In the next section, we will review the knowledge about t-product, orthogonal tensors, f-diagonal tensors, etc.

The T-SVD factorization (\ref{e1.2}) is used in low rank tensor approximation, where the rank is the tensor tubal rank \cite{ZLLZ18}.  To define the tensor tubal rank, the third order tensor $\A$ must go through Discrete Fourier transform (DFT), standard matrix SVD and inverse DFT.   The f-diagonal tensor $\S$, thus obtained, may be denoted as $\S = G(\A)$, where $G$ is the Kilmer-Martin mapping to represent the above processing.  In Section 3, we will state the Kilmer-Martin mapping in details.

The f-diagonal tensor $\S$ in this way obtained is a special f-diagonal tensor.  In order to  distinguish them from general f-diagonal tensors,  we call such an f-diagonal tensor an s-diagonal
tensor.   In Section 4, we formally define s-diagonal tensors:  an f-diagonal tensor $\S$ is an s-diagonal tensor if there is a tensor $\A$ such that $\S = G(\A)$.  Two $m \times n \times p$ tensors $\A$ and $\B$ are called orthogonally equivalent \cite{QLLO21} if there are an $m \times m \times p$ orthogonal tensor $\U$ and an $n \times n \times p$ orthogonal tensor $\V$ such that
$$\A = \U * \B * \V^\top.$$
It was proved in \cite{QLLO21} that if $\A$ and $\B$ are orthogonally equivalent, then $G(\A) = G(\B)$.  Based upon this result, in Section 4, we derive a checkable result for s-diagonal tensors: an f-diagonal tensor $\S$ is an s-diagonal tensor if and only if $G(\S) = \S$.   Then we have an orthogonal partition theorem for $T(m, n, p)$, the set of all $m \times n \times p$ real tensors. The third order tensor space is partitioned into orthogonal equivalence classes.   Each orthogonal equivalence class has a unique s-diagonal tensor. Any pair of  nonzero orthogonal equivalence classes have the same cardinality in the sense that a one-to-one relation can be established between these two classes.  Two s-diagonal tensors are equal if they are orthogonally equivalent.   Third order tensors in an orthogonal equivalence class have the same tensor tubal rank and T-singular values, which can be calculated from the entries of the s-diagonal tensor in that class.  The T-singular values of that class are linked with some extremal values of third order tensors in that class. The contributions of this paper are to reveal that the tensor obtained by Kilmer-Martin mapping is a special f-diagonal tensor, i.e, an s-diagonal tensor, and further identify the essential characteristics and properties of s-diagonal tensors.

In Section 5, we present four necessary conditions for an f-diagonal tensor to be an s-diagonal tensor.  These four necessary conditions are the tubal 2-norm decay property, the first frontal slice decay property, the third mode symmetry property and the tubal leading entry maximum property.   We state the meanings of these four necessary conditions in that section. In this way, sometimes we may identify some f-diagonal tensors are not s-diagonal tensors easily.   Some necessary conditions are especially useful. For example, the tubal 2-norm decay property was used in \cite{QY21} to define T-singular values and tail energy for error estimate of tensor sketching algorithms.

Then, in Section 6, we present a set of sufficient and necessary conditions for an f-diagonal tensor to be an s-diagonal tensor.  Such conditions involve a special complex number, hence are not so direct.   From these conditions, we conclude that the set of the s-diagonal tensors is a closed convex cone.

In Section 7, for $p = 2, 3, 4$, we present direct sufficient and necessary conditions for an f-diagonal tensor to be an s-diagonal tensor.   No complex numbers are involved in these conditions.  An s-diagonal tensor may not be nonnegative.  This can be seen from these sufficient and necessary conditions.

Some final remarks are made in Section 8.

\section{Preliminaries}

In this paper, $m, n$ and $p$ are positive integers, and $p \ge 2$.    Denote real matrices by capital Roman letters $A,B,\ldots$, complex matrices by capital Greek letters $\Delta, \Sigma, \ldots$, and tensors by Euler script letters $\mathcal{ A},\mathcal{B},\ldots$.    We use $\Re$ to denote the real number field, $\CC$ to denote the complex number field.
The set of all $m \times n \times p$ real tensors is denoted as $T(m, n, p)$.  Then $T(m, n, p)$ is a linear space.
For a third order tensor $\A \in T(m,n,p)$, its $(i,j,k)$-th element is represented by $\A(i, j, k)$.   The frontal slice $\mathcal{A}(:,:,k)$ is denoted by $A^{(k)}$.

A tensor $\A \in T(m, n, p)$ is called an f-diagonal tensor if all of its frontal slices $A^{(k)}$ are diagonal matrices for $k = 1, 2, \ldots, p$.   The set of all $m \times n \times p$ real f-diagonal tensors is denoted as $F(m, n, p)$, which is also a linear space.

For a third order tensor $\A \in T(m, n, p)$,  define
$${\rm bcirc}(\A):= \left(
\begin{array}{ccccc}
 A^{(1)}\ & A^{(p)} & A^{(p-1)} & \cdots & A^{(2)}\ \\
  A^{(2)} & A^{(1)} & A^{(p)} & \cdots & A^{(3)}\\
   \vdots\ \ \ & \ \vdots & \vdots\ \  & \ddots & \vdots\ \ \ \\
	A^{(p)} & A^{(p-1)} & A^{(p-2)} & \cdots & A^{(1)}
\end{array}\right),$$
 and bcirc$^{-1}($bcirc$(\A)):= \A$.    The block matrix bcirc$(\A)$ is an $mp \times np$ matrix.   If $m \not = n$, then we cannot say that bcirc$(\A)$ is symmetric.   However, we may say that bcirc$(\A)$ is symmetric in the sense of block if $A^{(k)} = A^{(p-k+2)}$ for $k = 2, 3, \ldots, 1 + \lfloor {p-1 \over 2}\rfloor$, where$\lfloor {p-1 \over 2}\rfloor$ represents the largest integer less than or equal to $ {p-1} \over 2$.

For $\A \in T(m, n, p)$, its transpose
 is defined as
 $$\A^\top = {\rm bcirc}^{-1}[({\rm bcirc}(\A))^\top].$$
 The identity tensor $\I_{nnp}$ is defined as
 $$\I_{nnp} = {\rm bcirc}^{-1}(I_{np}),$$
 where $I_{np}$ is the identity matrix in $\Re^{np \times np}$.

 For $\A \in T(m, n, p)$, define
 $${\rm unfold}(\A) := \left(\begin{array}{c} A^{(1)}\\ A^{(2)}\\  \vdots \\ A^{(p)}\end{array}\right) \in \Re^{mp \times n},$$
and fold$($unfold$(\A)) := \A$.   For $\A \in T(m, s, p)$ and $\B \in T(s, n, p)$, the T-product of $\A$ and $\B$ is defined as
$$\A * \B := {\rm fold}({\rm bcirc}(\A){\rm unfold}(\B)) \in T(m, n, p).$$

For $\U \in T(m, m, p)$, if  $\U * \U^\top = \U^\top * \U = \I_{mmp}$, then $\U$ is called an orthogonal tensor.   The set of all $m \times m \times p$ orthogonal tensors is denoted as $O(m, m, p)$, which is a group with respect to t-product.

\section{The Kilmer-Martin Mapping}

In this paper, we denote
$$\omega = e^{2\pi\sqrt{-1} \over p}.$$
This is a special complex number which will be involved in some of our discussion.
The $p \times p$ Discrete Fourier Transform (DFT) matrix $F_p$ has the form
$$F_p = \begin{bmatrix}
1 & 1 & \cdots & 1 & 1 \\
1 & \omega & \cdots  & \omega^{p-2} & \omega^{p-1} \\
\vdots & \vdots & \ddots& \vdots & \vdots \\
1 & \omega^{p-2} &  \cdots & \omega^{(p-2)(p-2)} & \omega^{(p-2)(p-1)}\\
1 & \omega^{p-1} &  \cdots & \omega^{(p-2)(p-1)} & \omega^{(p-1)(p-1)}
\end{bmatrix}.$$
The conjugate transpose of $F_p$ is denoted as $F_p^*$, the inverse of $F_p$ is $F_p^{-1}=\dfrac{1}{p}F_p^*$.

For $\A \in T(m, n, p)$, we may block-diagonalize bcirc$(\A)$ as
$$\Delta(\A) := (F_p \otimes I_m) {\rm bcirc}(\A) (F_p^{-1} \otimes I_n) = \begin{bmatrix}
\Delta^{(1)} &  &  & \\
 & \Delta^{(2)} &  & \\
 &  & \ddots& \\
 & &  & \Delta^{(p)} \\
\end{bmatrix},$$
where $\otimes$ denotes the Kronecker product, $\Delta^{(k)} \in \CC^{m \times n}$ for $k = 1, 2, \ldots, p$.
Then we have
\begin{equation} \label{e3.3-1}
\Delta^{(k)} \equiv \Delta^{(k)}(\A) = \sum_{l=1}^p \omega^{(l-1)(k-1)} A^{(l)},
\end{equation}
for $k = 1, 2, \ldots, p$.

For each matrix $\Delta^{(k)}$, compute its SVD
$$\Delta^{(k)} = \Phi^{(k)}D^{(k)}\left({\Psi^{(k)}}\right)^*,$$
where $\Phi^{(k)} \in \CC^{m \times m}$ and $\Psi^{(k)} \in \CC^{n \times n}$ are unitary matrices, $D^{(k)} \in \Re^{m \times n}$ is a nonnegative diagonal matrix, the singular values of $\Delta^{(k)}$ follow the standard nonincreasing order.  Denote
$$D(\A) :=  \begin{bmatrix}
D^{(1)} &  &  & \\
 & D^{(2)} &  & \\
 &  & \ddots& \\
 & &  & D^{(p)} \\
\end{bmatrix}.$$
We also denote the $i$th diagonal entry of $D^{(k)}$ as $\D(i, i, k)$.

Let
$$\S = G(\A) :=  {\rm bcirc}^{-1}\left((F_p^{-1} \otimes I_m)D(\A)(F_p \otimes I_n)\right).$$
Then $\S = G(\A) \in F(m, n, p)$.  In particular, by the above equality, we have
\begin{equation} \label{e2.2}
\S(i, i, k) = {1 \over p}\sum_{l=1}^p \bar \omega^{(k-1)(l-1)}\D(i, i, l),
\end{equation}
for $i = 1, 2,  \ldots, \min \{ m, n \}$, $k = 1, 2, \ldots, p$.

We call $G(\cdot)$ the Kilmer-Martin mapping.

Let
$$\Phi(\A) :=  \begin{bmatrix}
\Phi^{(1)} &  &  & \\
 & \Phi^{(2)} &  & \\
 &  & \ddots& \\
 & &  & \Phi^{(p)} \\
\end{bmatrix},$$
$$\Psi(\A) :=  \begin{bmatrix}
\Psi^{(1)} &  &  & \\
 & \Psi^{(2)} &  & \\
 &  & \ddots& \\
 & &  & \Psi^{(p)} \\
\end{bmatrix},$$
$$\U = \U(\A) = {\rm bcirc}^{-1}\left((F_p^{-1} \otimes I_m)\Phi(\A)(F_p \otimes I_n)\right),$$
$$\V = \V(\A) = {\rm bcirc}^{-1}\left((F_p^{-1}\otimes I_m)\Psi(\A)(F_p \otimes I_n)\right).$$
Then $\U \in O(m, m, p)$ and $\V \in O(n, n, p)$, and we have
\begin{equation} \label{e3.1}
\A = \U * \S * \V^\top.
\end{equation}

By \cite{KM11, KMP08}, we have
\begin{equation} \label{e4.1}
 \sum_{k=1}^p \S(1, 1, k)^2 \ge \sum_{k=1}^p \S(2, 2, k)^2 \ge \ldots  \ge  \sum_{k=1}^p \S(\min \{ m,n \}, \min \{ m,n \}, k)^2.
\end{equation}

In \cite{QY21}, the $i$th largest T-singular value of $\A$ is defined as
$$\sigma_i := \sqrt{\sum_{k=1}^p \S(i, i, k)^2},$$
for $i = 1, 2,  \ldots, \min \{ m, n \}$.   T-singular values are used there to define the tail energy for the error estimate of a proposed tensor sketching algorithm.

\section{s-Diagonal Tensors}

\begin{definition}   Denote the Kilmer-Martin mapping as $G(\cdot)$.   Let $\S \in F(m, n, p)$.   If there is $\A \in T(m, n, p)$ such that $G(\A) = \S$,  then we say that $\S$ is an s-diagonal tensor.    The set of all $m \times n \times p$ s-diagonal tensors is denoted as $S(m, n, p)$.   Let $\A \in T(m, n, p)$.  If there exist $\U \in O(m, m, p)$, $\V \in O(n, n, p)$ and $\S \in S(m, n, p)$ such that
\begin{equation} \label{e4.5-2}
\A = \U * \S * \V^\top,
\end{equation}
then we say that $\A$ has an ST-SVD (Standard T-SVD) factorization (\ref{e4.5-2}).
\end{definition}

Note that in (\ref{e4.5-2}), we do not require that $\U$ and $\V$ are resulted from the Kilmer-Martin process, while (\ref{e3.1}) shows that such $\U$ and $\V$ always exist as long as $\S = G(\A)$.


The following theorem is Theorem 3.2 of \cite{QLLO21}.

\begin{theorem} \label{t4.2}
Suppose that $\A, \B \in T(m, n, p)$ and they are orthogonally equivalent.  Then $G(\A) = G(\B)$.
\end{theorem}

From this theorem, we have a checkable condition to determine a given tensor in $F(m, n, p)$ is in $S(m, n, p)$ or not.

\begin{theorem} \label{t4.3}
Suppose that $\S \in F(m, n, p)$.  Then $\S \in S(m, n, p)$ if and only if $G(\S) = \S$.
\end{theorem}
\begin{proof} If $G(\S) = \S$, then by definition, $\S \in S(m, n, p)$.   On the other hand, suppose $\S \in S(m, n, p)$.  Then there is $\A \in T(m, n, p)$ such that $\S = G(\A)$.   There exist $\U \in O(m, m, p)$ and $\V \in O(n, n, p)$ such that (\ref{e3.1}) holds.   Then $\A$ and $\S$ are orthogonally equivalent.  By Theorem \ref{t4.2}, we have $G(\S) = G(\A) = \S$.
\end{proof}

The orthogonal equivalence is an equivalence relation of $T(m, n, p)$.   We have the following orthogonal partition theorem.

\begin{theorem} \label{t4.4}
The linear space $T(m, n, p)$ is partitioned to equivalence classes by the orthogonal equivalence relation.  Each orthogonal equivalence class has a unique s-diagonal tensor. All the nonzero orthogonal equivalence classes have the same cardinality in the sense that a one-to-one relation can be established between any two such classes. Two s-diagonal tensors are equal if they are orthogonally equivalent.   Third order tensors in an orthogonal equivalence class have the same tensor tubal rank and T-singular values, which can be calculated from the entries of the s-diagonal tensor in that class.
\end{theorem}
\begin{proof}  It is easy to see that the orthogonal equivalence is an equivalence relation of $T(m, n, p)$.
Then we have the first conclusion.  By Theorem \ref{t4.2}, two s-diagonal tensors are equal if they are orthogonally equivalent.   Then each orthogonal equivalence class has a unique s-diagonal tensor.  Suppose two orthogonal equivalence classes of nonzero tensors have their s-diagonal tensors $\S_1$ and $\S_2$.  Then any tensor in these two classes have the forms
$$\A_1 = \U * \S_1 * \V^\top$$
and
$$\A_2 = \U * \S_2 * \V^\top,$$
respectively, where $\U \in O(m, m, p)$ and $\V \in O(n, n, p)$.   Then there is a one-to-one relation  $\A_1 \leftrightarrow \A_2$ between these two classes, and the two classes have the same cardinality.  By the definitions of tensor tubal rank and T-singular values, we have the remaining conclusions.
\end{proof}

The following proposition was proved in \cite{QY21}.

\begin{proposition} \label{p4.5}
Suppose that $\A \in T(m, n, p)$ has T-singular values
$$\sigma_1 \ge \sigma_2 \ge \ldots \ge \sigma_{\min \{m, n\}},$$
and $1 \le s \le \min \{m, n\}$.   Then
$$\sum_{l=1}^s \sigma_l^2 \ge \sum_{l=1}^s \sum_{k=1}^p \A(i_l, j_l, k)^2,$$
where $1 \le i_l \le m$, $1 \le j_l \le n$, $(i_1, j_1), \ldots, (i_s, j_s)$ are distinct pairs of indices.
\end{proposition}

By Theorems \ref{t4.2} and \ref{t4.4}, Proposition \ref{p4.5} indicates that the T-singular values of that class are linked with some extremal values of third order tensors in that class.

\begin{proposition}
If there are $\A \in T(m, n, p)$, $\U \in O(m, m, p)$, $\V \in O(n, n, p)$ and $\S \in S(m, n, p)$, such that
\begin{equation} \label{e4.5-1}
\A = \U * \S * \V^\top,
\end{equation}
then $\S = G(\A)$ and (\ref{e4.5-1}) is an ST-SVD of $\A$.
\end{proposition}
\begin{proof} By (\ref{e4.5-1}), $\A$ and $\S$ are orthogonally equivalent.  Then $G(\A) = G(\S) = \S$ as $\S \in S(m, n, p)$.
\end{proof}

Thus, if we have (\ref{e4.5-1}) and some sufficient conditions to show that $\S \in S(m, n, p)$, then we may conclude $\S = G(\A)$ without calculate $G(\A)$.   In Sections 6 and 7, we will present some sufficient and necessary conditions to determine if $\S \in S(m, n, p)$, for a given f-diagonal tensor $\S$.

\section{Necessary Conditions for s-Diagonal Tensors}

There are four major features of  s-diagonal tensors.   Suppose that $\S \in S(m, n, p)$.   Then $\S$ has the following four properties:

\medskip

{\bf (1) Tubal 2-Norm Decay Property}, namely (\ref{e4.1}).
The 2-norms of the tubal vectors $\S(i, i, :)$ decay with $i = 1, 2,  \ldots, \min \{ m,n \}$.  It was used in \cite{QY21} to define T-singular values and tail energy for error estimate of tensor sketching algorithms.
This property can be found in \cite{KMP08}.  A proof for it can be found in \cite{LHPQ21}.

\medskip

{\bf (2) First Frontal Slice Decay Property}, namely
\begin{equation} \label{e4.2}
 \S(1, 1, 1) \ge  \S(2, 2, 1) \ge \ldots  \ge  \S(\min \{ m,n \}, \min \{ m,n \}, 1) \ge 0.
\end{equation}
Though some negative entries may appear in $\S$, the entries of its first frontal slice must be nonnegative and have a decay property.   This property can be found in \cite{Lu18}.  A proof for it can be found in \cite{LHPQ21}.

\medskip

{\bf (3) Third Mode Symmetry Property}, namely
\begin{equation} \label{e4.2a}
\S(i, i, k) = \S(i, i, p-k+2),
\end{equation}
for $i = 1, 2,  \ldots, \min \{ m, n \}$ and $k = 2, 3, \ldots, 1 + \lfloor {p-1 \over 2}\rfloor$.  The tensor $\S$ is symmetric for the third mode in the above sense.  Equivalently, this condition may be stated as that bcirc$(\S)$ is symmetric in the sense of block.
We identify this property in this paper and will prove it in the following theorem.

\medskip

{\bf (4) Tubal Leading Entry Maximum Property}, namely
\begin{equation} \label{e4.3}
\S(i, i, 1) \ge |\S(i, i, k)|,
\end{equation}
for $i = 1, 2,  \ldots, \min \{ m, n \}$ and $k = 1, 2, \ldots, p$.   For each tube of $\S$, its first entry takes the maximum value.    We identify this property in this paper and will prove it in the following theorem.

\medskip

We have the following theorem.

\begin{theorem}   \label{t4.1}
Suppose that $\S \in S(m, n, p)$.   Then $\S$ has the tubal 2-norm decay property (\ref{e4.1}), the
first frontal slice decay property (\ref{e4.2}), the third mode symmetry property (\ref{e4.2a}), and the
tubal leading entry maximum property (\ref{e4.3}).
\end{theorem}
\begin{proof} As stated above, the tubal 2-norm decay property (\ref{e4.1}) and the
first frontal slice decay property (\ref{e4.2}) have been identified before and proved elsewhere.

We now prove the third mode symmetry property (\ref{e4.2a}).  Since $\S \in S(m, n, p)$, there is $\A \in T(m, n, p)$ such that $\S = G(\A)$.   Let $\D = D(\A)$.  Then by (\ref{e2.2}),
\begin{equation} \label{e4.6}
\S(i, i, k) = {1 \over p}\sum_{l=1}^p \bar \omega^{(k-1)(l-1)}\D(i, i, l),
\end{equation}
for $i = 1, 2, \ldots, \min \{ m, n \}$, $k = 1, 2, \ldots, p$.   Then for $k = 2, 3, \ldots, 1 + \lfloor {p-1 \over 2}\rfloor$,
\begin{eqnarray*}
\S(i, i, p-k+2) & = & {1 \over p}\sum_{l=1}^p \bar \omega^{(p-k+2-1)(l-1)}\D(i, i, l)\\
& = & {1 \over p}\sum_{l=1}^p \bar \omega^{(1-k)(l-1)}\D(i, i, l).
\end{eqnarray*}
Since $\bar \omega^{(k-1)(l-1)}$ and $\bar \omega^{(1-k)(l-1)}$ are conjugate to each other, they have the same real part.   Thus, $\S(i, i, k)$ and $\S(i, i, p-k+2)$ have the same real part.   However, $\S(i, i, k)$ and $\S(i, i, p-k+2)$ are real, thus, they are equal.  We have (\ref{e4.2a}).

We then prove the tubal leading entry maximum property (\ref{e4.3}).  By (\ref{e4.6}),
for $i = 1, 2,  \ldots, \min \{ m, n \}$, $k = 1, 2, \ldots, p$, we have
\begin{eqnarray*}
|\S(i, i, k)| & = & {1 \over p} \left|\sum_{l=1}^p \bar \omega^{(k-1)(l-1)}\D(i, i, l)\right|\\
& \le & {1 \over p} \sum_{l=1}^p |\bar \omega^{(k-1)(l-1)}||\D(i, i, l)|\\
& \le & {1 \over p} \sum_{l=1}^p |\D(i, i, l)|\\
& = & {1 \over p} \sum_{l=1}^p \D(i, i, l)\\
& = & \S(i, i, 1).
\end{eqnarray*}
This proves (\ref{e4.3}).
\end{proof}

However, the tubal 2-norm decay property (\ref{e4.1}), the
first frontal slice decay property (\ref{e4.2}), the third mode symmetry property (\ref{e4.2a}), and the
tubal leading entry maximum property (\ref{e4.3}) are only necessary conditions for s-diagonal tensors.    For example, let $m = n = p = 3$, $\S(1, 1, 1) = 12$, $\S(2, 2, 1) = 8$, $\S(3, 3, 1) = 5$,
$\S(1, 1, 2) = \S(1, 1, 3) = 5$, the other entries of $\S$ are zero.   Then the four properties (\ref{e4.1}) and (\ref{e4.2}-\ref{e4.3}) are satisfied.  But we have $G(\S) \not = \S$.  By Theorem \ref{t4.3}, $\S \not \in S(3, 3, 3)$.  Thus, the four properties (\ref{e4.1}) and (\ref{e4.2}-\ref{e4.3}) are only necessary conditions for s-diagonal tensors.

\section{Sufficient and Necessary Conditions for s-Diagonal Tensors}

We now present a set of sufficient and necessary conditions for s-diagonal tensors.

\begin{theorem} \label{t6.1}
Suppose that $\S \in F(m, n, p)$ satisfies the third mode symmetry property (\ref{e4.2a}). Then $\Delta \equiv \Delta(\S)$ is real and diagonal.   This implies that for $i = 1, 2,  \ldots, \min \{ m, n \}$ and $k = 1, 2, \ldots, p$,
\begin{equation} \label{e5.8}
\sum_{l=1}^p \omega^{(l-1)(k-1)} \S(i, i, l) \equiv \Delta_{ii}^{(k)} \in \Re.
\end{equation}

Then $\S \in S(m, n, p)$ if and only if
for $i = 1, 2,  \ldots, \min \{ m, n \}$ and $k = 1, 2, \ldots, p$, we have
\begin{equation}  \label{e5.9}
\sum_{l=1}^p \omega^{(l-1)(k-1)} \S(i, i, l) \ge 0,
\end{equation}
and
\begin{equation} \label{e5.10}
\sum_{l=1}^p \omega^{(l-1)(k-1)} [\S(i, i, l)-\S(i+1,i+1, l)] \ge 0,
\end{equation}
where $\S(i+1,i+1, l) = 0$ for $i = \min \{ m, n \}$.
\end{theorem}
\begin{proof}  Since $\S \in F(m, n, p)$, $\S^{(k)}$ is diagonal for $k = 1, 2, \ldots, p$.  By (\ref{e3.3-1}),
$$\Delta^{(k)} \equiv \Delta^{(k)}(\S) = \sum_{l=1}^p \omega^{(l-1)(k-1)} \S^{(l)}.$$
Then, $\Delta^{(k)}$ is diagonal for $k = 1, 2, \ldots, p$.    We now show that $\Delta^{(k)}$ is real for $k = 1, 2, \ldots, p$.   For $i = 1, 2,  \ldots, \min \{ m, n \}$ and $k = 1, 2, \ldots, p$, by (\ref{e2.2}), we have
\begin{eqnarray*}
\Delta_{ii}^{(k)} & = & \sum_{l=1}^p  \omega^{(l-1)(k-1)}\S(i, i, l)\\
& = & \S(i, i, 1) + \sum_{l=2}^p  \omega^{(l-1)(k-1)}\S(i, i, l),
\end{eqnarray*}
where $\S(i, i, 1)$ is real.
By the third mode symmetry property (\ref{e4.2a}),
$$\omega^{(p-l+2-1)(k-1)}\S(i,i, p-l+2) = \omega^{(1-l)(k-1)}\S(i, i, l)$$
is conjugate with $\omega^{(l-1)(k-1)}\S(i, i, l)$.  Thus,
$$\omega^{(l-1)(k-1)}\S(i, i, l) + \omega^{(p-l+2-1)(k-1)}\S(i,i, p-l+2) \in \Re.$$
If $p$ is odd, we have
$$\Delta_{ii}^{(k)} = \S(i, i, 1) + \sum_{l=2}^{p+1 \over 2} \left[\omega^{(l-1)(k-1)}\S(i, i, l) + \omega^{(p-l+2-1)(k-1)}\S(i,i, p-l+2) \right] \in \Re.$$
If $p$ is even, note that $\omega^{p \over 2} = e^{{2\pi \over p}\cdot {p \over 2}\sqrt{-1}}= e^{\pi\sqrt{-1}} = -1$.  Thus,
	\begin{equation*}
\begin{split}
			\Delta_{ii}^{(k)} &=  \S(i, i, 1) +(-1)^{k-1}\S\left(i, i, {p\over 2}+1\right) +\\ &\sum_{l=2}^{p \over 2} \left[\omega^{(l-1)(k-1)}\S(i, i, l) + \omega^{(p-l+2-1)(k-1)}\S(i,i, p-l+2) \right] \in \Re.
	\end{split}
	\end{equation*}
Therefore, $\Delta_{ii}^{(k)}$ is real.  Hence $\Delta = \Delta(\S)$ is real and diagonal.

Now, if (\ref{e5.9}) and (\ref{e5.10}) hold, then for $k = 1, 2, \ldots, p$, we have
$$\Delta_{11}^{(k)} \ge \Delta_{22}^{(k)} \ge \ldots \ge \Delta_{\min \{ m,n \}, \min \{ m,n \}}^{(k)} \ge 0.$$
Since $\Delta^{(k)}$ is diagonal, this implies that $\Delta^{(k)} = D^{(k)}$ for $k = 1, 2, \ldots, p$, i.e., $G(\S) = \S$.  Thus, $\S \in S(m, n, p)$.

On the other hand, if $\S \in S(m, n, p)$, then $G(\S) = \S$.    Thus, for $k = 1, 2, \ldots, p$,
$$\Delta^{(k)} = \sum_{l=1}^p \omega^{(l-1)(k-1)} S^{(l)} = D^{(k)}.$$
Hence, $\Delta^{(k)}$ is real nonnegative diagonal matrix, whose diagonal entries follow a nonincreasing order.  This implies (\ref{e5.9}) and (\ref{e5.10}).
\end{proof}

By this theorem, we may conclude that $S(m, n, p)$ is a closed convex cone.  The question is what is its dual cone?

The conditions (\ref{e5.9}) and (\ref{e5.10}) involve the complex number $\omega = {2\pi\sqrt{-1} \over p}$.   Their meanings are not so direct comparing with conditions (\ref{e4.1}) and (\ref{e4.2}-\ref{e4.3}).  Thus, we still wish to find direct sufficient and necessary conditions for s-diagonal tensors.

\section{Direct Sufficient and Necessary Conditions In the Cases that $p=2, 3, 4$}

In this section, we present direct sufficient and necessary conditions for s-diagonal tensors in the cases that $p = 2, 3, 4$.

\subsection{The Case that $p=2$}

When $p=2$, the third mode symmetry property (\ref{e4.2a}) does not exist, and the
tubal leading entry maximum property (\ref{e4.3}) is still needed.  For $p=2$, the
tubal leading entry maximum property (\ref{e4.3}) has the form
\begin{equation} \label{e6.11}
\S(i, i, 1) \ge |\S(i, i, 2)|,
\end{equation}
for $i = 1, 2, \ldots, \min \{ m, n \}$.  Then we may replace the tubal 2-norm decay property (\ref{e4.1}) and the first frontal slice decay property (\ref{e4.2}) by a strong condition

\medskip

{\bf (5) Strong First Frontal Slice Decay Property}, namely
\begin{equation} \label{e6.12}
\S(i, i, 1) - \S(i+1, i+1, 1) \ge |\S(i, i, 2) - \S(i+1,i+1,2)|,
\end{equation}
for $i = 1, 2, \ldots, \min \{ m, n \}-1$.
We call this property the strong first frontal slice decay property, as it is stronger than the first frontal slice decay property (\ref{e4.2}) under the tubal leading entry maximum property (\ref{e6.11}).
Actually, for $p=2$, (\ref{e6.11}) and (\ref{e6.12}) imply (\ref{e4.1}) and (\ref{e4.2}).

\medskip

We have the following theorem.

\begin{theorem} \label{t7.1}
Let $\S \in F(m, n, 2)$.   Then $\S \in S(m, n, 2)$ if and only if the tubal leading entry maximum property (\ref{e6.11}) and the strong first frontal slice decay property (\ref{e6.12}) hold.
\end{theorem}
\begin{proof}  For $p=2$, $\omega = e^{2\pi\sqrt{-1} \over 2} = -1$.  Since $p=2$,
$$\Delta_{ii}^{(1)} \equiv \Delta_{ii}^{(1)}(\S) = \sum_{l=1}^2 \omega^0 \S(i, i, l) = \S(i, i, 1) + \S(i, i, 2) \in \Re,$$
$$\Delta_{ii}^{(2)} \equiv \Delta_{ii}^{(2)}(\S) = \sum_{l=1}^2 \omega^{l-1} \S(i, i, l) = \S(i, i, 1) - \S(i, i, 2) \in \Re.$$
The equality (\ref{e6.11}) is equivalent to
$$\S(i, i, 1) \ge \max \{ \S(i, i, 2), -\S(i, i, 2) \},$$
i.e.,
$$\S(i,i,1) + \S(i,i,2) \ge 0,\ {\rm and}\ \ \S(i,i,1) - \S(i,i,2)\ge 0,$$
for $i = 1, 2, \ldots, \min \{ m, n \}$.    This is (\ref{e5.9}) for $p=2$.

The equality (\ref{e6.12})  is equivalent to
$$\S(i,i,1)-\S(i+1,i+1,1) \ge \max \{ \S(i,i,2)-\S(i+1,i+1,2), \S(i+1,i+1,2)-\S(i,i,2)\},$$
i.e.,
$$\left[\S(i,i,1)-\S(i+1,i+1,1)\right]+\left[\S(i,i,2)-\S(i+1,i+1,2)\right]\ge 0$$
and
$$\left[\S(i,i,1)-\S(i+1,i+1,1)\right]-\left[\S(i,i,2)-\S(i+1,i+1,2)\right]\ge 0,$$
for $i = 1, 2, \ldots, \min \{ m, n \}-1$.  This is (\ref{e5.10}) for $p=2$.

The conclusion follows from Theorem \ref{t6.1}.
\end{proof}

Theorem \ref{t7.1} indicates that an s-diagonal tensor may not be nonnegative.

\subsection{The Case that $p=3$}

When $p=3$, the third mode symmetry property (\ref{e4.2a}) has the form
\begin{equation} \label{e7.12}
\S(i,i,2) = \S(i,i,3),
\end{equation}
for $i = 1, 2, \ldots, \min \{ m, n \}$.   Then we need two more conditions.

\medskip

{\bf (6) Strong Tubal Leading Entry Maximum Property}, namely
\begin{equation} \label{e7.13}
\S(i, i, 1) \ge \max \{-2\S(i,i,2), \S(i,i,2)\}
\end{equation}
for $i = 1, 2,  \ldots, \min \{ m, n \}$.  For $p=2$, this condition is stronger than  the tubal leading entry maximum property (\ref{e4.3}).  Thus, we call it the strong tubal leading entry maximum property.

\medskip

\medskip

{\bf (7) Strong First Frontal Slice Decay Property for $p=3$}, namely
\begin{eqnarray}
&& \S(i, i, 1) - \S(i+1, i+1, 1)\nonumber \\
& \ge & \max \left\{ -2[\S(i, i, 2) - \S(i+1,i+1,2)], \S(i, i, 2) - \S(i+1,i+1,2) \right\},\label{e7.14}
\end{eqnarray}
for $i = 1, 2, \ldots, \min \{ m, n \}-1$.
We call this property the strong first frontal slice decay property for $p=3$, as it is stronger than the first frontal slice decay property (\ref{e4.2}) under the strong tubal leading entry maximum property (\ref{e7.13}).

\medskip

We have the following theorem.

\begin{theorem}
Let $\S \in F(m, n, 3)$.   Then $\S \in S(m, n, 3)$ if and only if  the third mode symmetry property (\ref{e7.12}), the strong tubal leading entry maximum property (\ref{e7.13}) and the strong first frontal slice decay property for $p=3$ (\ref{e7.14}) hold.
\end{theorem}
\begin{proof} For $p=3$, we have $\omega = -{1 \over 2}+ {\sqrt{-3}\over 2}$ and $\omega^2 = -{1 \over 2}- {\sqrt{-3}\over 2}$.  Then for $i = 1, 2,  \ldots, \min \{ m, n \}$,
$$\Delta_{ii}^{(1)} \equiv \Delta_{ii}^{(1)}(\S) = \sum_{l=1}^3 \S(i,i,l) = \S(i,i,1) + 2\S(i,i,2),$$
as by (\ref{e7.12}), $\S(i,i,2) = \S(i,i,3)$,
\begin{eqnarray*}
\Delta_{ii}^{(2)} \equiv \Delta_{ii}^{(2)}(\S) & = & \sum_{l=1}^3 \omega^{l-1}\S(i,i,l)\\
& = & \S(i,i,1)+ \omega\S(i,i,2) + \omega^2\S(i,i,3)\\
& = & \S(i,i,1)-\S(i,i,2),
\end{eqnarray*}
\begin{eqnarray*}
\Delta_{ii}^{(3)} \equiv \Delta_{ii}^{(3)}(\S) & = & \sum_{l=1}^3 \omega^{2(l-1)}\S(i,i,l)\\
& = & \S(i,i,1)+ \omega^2\S(i,i,2) + \omega\S(i,i,3)\\
& = & \S(i,i,1)-\S(i,i,2).
\end{eqnarray*}
Under the condition (\ref{e7.12}), the inequalities (\ref{e7.13}) and (\ref{e7.14}) are (\ref{e5.9}) and (\ref{e5.10}), respectively, for $p=3$.   Together with the third mode symmetry property (\ref{e7.12}), we have the conclusion by Theorem \ref{t6.1}.
\end{proof}

\subsection{The Case that $p=4$}

When $p=4$, the third mode symmetry property (\ref{e4.2a}) has the form
\begin{equation} \label{e7.15}
\S(i,i,2) = \S(i,i,4),
\end{equation}
for $i = 1, 2, \ldots, \min \{ m, n \}$.   Again, we need two more conditions:
\begin{equation} \label{e7.16}
\S(i,i,1) + \S(i,i,3) \ge \max \{ 2|\S(i,i,2)|, 2\S(i,i,3) \},
\end{equation}
for $i = 1, 2, \ldots, \min \{ m, n \}$,
and
$$[\S(i,i,1)- \S(i+1,i+1,1)] + [\S(i,i,3)-\S(i+1,i+1,3)]$$
\begin{equation} \label{e7.17}
\ge \max \{ 2|\S(i,i,2)-\S(i+1,i+1,2)|, 2[\S(i,i,3)-\S(i+1,i+1,3)]\},
\end{equation}
for $i = 1, 2, \ldots, \min \{ m, n \}-1$.  We will see that under (\ref{e7.15}), (\ref{e7.16}) and (\ref{e7.17}) are (\ref{e5.9}) and (\ref{e5.10}), respectively, for $p=4$.

\medskip

We have the following theorem.

\begin{theorem}
Let $\S \in F(m, n, 4)$.   Then $\S \in S(m, n, 4)$ if and only if (\ref{e7.15}), (\ref{e7.16}) and (\ref{e7.17}) hold.
\end{theorem}
\begin{proof}  For $p=4$, we have $\omega = \sqrt{-1}$, $\omega^2 = -1$, $\omega^3 = -\sqrt{-1}$ and $\omega^4 = 1$.  Then, for $i = 1, 2,  \ldots, \min \{ m, n \}$, we have
$$\Delta_{ii}^{(1)} \equiv \Delta_{ii}^{(1)}(\S) = \sum_{l=1}^4 \S(i,i,l) = \S(i,i,1) + 2\S(i,i,2)+ \S(i,i,3),$$
\begin{eqnarray*}
\Delta_{ii}^{(2)} \equiv \Delta_{ii}^{(2)}(\S) & = & \sum_{l=1}^4 \omega^{l-1}\S(i,i,l)\\
& = & \S(i,i,1)+ \omega\S(i,i,2) + \omega^2\S(i,i,3) + \omega^3\S(i,i,4)\\
& = & \S(i,i,1)-\S(i,i,3),
\end{eqnarray*}
\begin{eqnarray*}
\Delta_{ii}^{(3)} \equiv \Delta_{ii}^{(3)}(\S) & = & \sum_{l=1}^4 \omega^{2(l-1)}\S(i,i,l)\\
& = & \S(i,i,1)+ \omega^2\S(i,i,2) + \omega^4\S(i,i,3) + \omega^6\S(i,i,4)\\
& = & \S(i,i,1)-2\S(i,i,2)+\S(i,i,3),
\end{eqnarray*}
and
\begin{eqnarray*}
\Delta_{ii}^{(4)} \equiv \Delta_{ii}^{(4)}(\S) & = & \sum_{l=1}^4 \omega^{3(l-1)}\S(i,i,l)\\
& = & \S(i,i,1)+ \omega^3\S(i,i,2) + \omega^6\S(i,i,3) + \omega^9\S(i,i,4)\\
& = & \S(i,i,1)-\S(i,i,3).
\end{eqnarray*}
Under (\ref{e7.15}), (\ref{e7.16}) and (\ref{e7.17}) are (\ref{e5.9}) and (\ref{e5.10}), respectively, for $p=4$.  Then, combining with (\ref{e7.15}), we have the conclusion by Theorem  \ref{t6.1}.
\end{proof}

\medskip
We see that direct sufficient and necessary conditions for s-diagonal tensors become more and more complicated as $p$ increases.

\section{Final Remarks}

In this paper, we reveal that the f-diagonal tensor resulted from the Kilmer-Martin mapping is a special f-diagonal tensor.   We call such an f-diagonal tensor an s-diagonal tensor, present a checkable condition, four necessary conditions and some sufficient and necessary conditions for an f-diagonal tensor to be an s-diagonal tensor or not.   The checkable condition involves the Kilmer-Martin mapping $G$.
Theorem \ref{t4.4} indicates that $G^{-1}(\S)$ for $\S \in S(m, n, p)$ partition $T(m, n, p)$ to equivalence classes.
An $m \times n \times p$ f-diagonal tensor $\S$ is an s-diagonal tensor if and only if $G(\S) = \S$.  The four necessary conditions, and the sufficient and necessary conditions for $p=2, 3, 4$, only involve some algebraic equality and inequalities of the entries of $\S$.  The general sufficient and necessary conditions involve a complex number $\omega = {2\pi\sqrt{-1} \over p}$.  These results enrich our understanding to T-SVD factorization and low tensor tubal rank approximation, and will be useful for their further applications.

There are two main further exploration directions on this research topic.

(1) Find sufficient and necessary conditions without involving $\omega$ for $p \ge 5$.

(2) For a given s-diagonal tensor $\S$, find properties related with $\A \in G^{-1}(\S)$.    

Martin, Shafer and LaRue \cite{MSL13} extended T-SVD factorization to higher order tensors.  The discussion in this paper may also be considered to extended to higher order tensors.

\vskip2mm


          \end{document}